\documentclass[12pt,reqno]{article}

\usepackage[usenames]{color}
\usepackage{amssymb}
\usepackage{amsmath}
\usepackage{amsthm}
\usepackage{amsfonts}
\usepackage{amscd}
\usepackage{graphicx}

\usepackage[colorlinks=true,
linkcolor=webgreen,
filecolor=webbrown,
citecolor=webgreen]{hyperref}

\definecolor{webgreen}{rgb}{0,.5,0}
\definecolor{webbrown}{rgb}{.6,0,0}

\usepackage{fullpage}
\usepackage{float}

\usepackage{graphics}
\usepackage{latexsym}
\usepackage{epsf}
\usepackage{breakurl}

\usepackage{upquote}
\usepackage{listings}

\newcommand{\seqnum}[1]{\href{https://oeis.org/#1}{\rm \underline{#1}}}

\newcommand{\abs}[1]{\left\lvert#1\right\rvert}
\newcommand{\Nset}{\mathbb{N}}
\newcommand{\Zset}{\mathbb{Z}}
\newcommand{\Qset}{\mathcal{Q}}
\newcommand{\Sset}{\mathcal{S}}
\newcommand{\Dset}{\mathcal{D}}
\newcommand{\rb}[1]{_{\rm #1}}
\newcommand{\down}[1]{\bar{#1}}
\newcommand{\floor}[1]{\lfloor#1\rfloor}

\newcommand{\dsfrac}[2]{\frac{\displaystyle{#1}}{\displaystyle{#2}}} 
\newlength{\defaultmultlinegap}
\newcommand{\zmlg}{\setlength{\multlinegap}{0pt}}
\newcommand{\rmlg}{\setlength{\multlinegap}{\defaultmultlinegap}}
\newcommand{\loi}{{l_{\mathrm{o}i}}}
\newcommand{\loei}{{l_{\mathrm{oe}i}}}
\newcommand{\tloi}{{\tilde{l}_{\mathrm{o}i}}}
\newcommand{\koi}{{k_{\mathrm{o}i}}}
\newcommand{\koei}{{k_{\mathrm{oe}i}}}
\newcommand{\tkoi}{{\tilde{k}_{\mathrm{o}i}}}
\renewcommand{\le}[1]{l_{\mathrm{e}#1}}
\renewcommand{\ker}{{k_{\mathrm{e}r}}}
\newcommand{\lci}{{l_{\mathrm{c}i}}}
\newcommand{\tlc}[1]{\tilde{l}_{\mathrm{c}#1}}
\newcommand{\kci}{{k_{\mathrm{c}i}}}
\newcommand{\tkci}{{\tilde{k}_{\mathrm{c}i}}}
\newcommand{\lce}[1]{l_{\mathrm{ce}#1}}
\newcommand{\kcei}{{k_{\mathrm{ce}i}}}
\newcommand{\negspc}{\hspace{0.4ex}}
\newcommand{\qp}{\negspc+\negspc}
\newcommand{\qm}{\negspc-\negspc}
\newcommand{\Gpath}[3]{\operatorname{\mathit{G}_{#1\to #2}^{\{#3\}}}}
\newcommand{\Gloop}[2]{\operatorname{\mathit{G}_{#1}^{\{#2\}}}}
\newcommand{\mat}[1]{\mathbf{#1}}
\begin{document}

\begin{center}
\vspace*{5mm}
\end{center}

\theoremstyle{plain}
\newtheorem{theorem}{Theorem}
\newtheorem{corollary}[theorem]{Corollary}
\newtheorem{lemma}[theorem]{Lemma}
\newtheorem{proposition}[theorem]{Proposition}

\theoremstyle{definition}
\newtheorem{definition}[theorem]{Definition}
\newtheorem{example}[theorem]{Example}

\theoremstyle{remark}
\newtheorem{remark}[theorem]{Remark}

\begin{center}
\vskip 1cm{\LARGE\bf
  Connections Between Combinations Without\\
  Specified Separations and Strongly Restricted\\
\vskip .08in 
  Permutations, Compositions, and Bit Strings
}
\vskip 1cm
\large
Michael A. Allen\\
Physics Department\\
Faculty of Science\\
Mahidol University\\
Rama 6 Road\\
Bangkok 10400\\
Thailand\\
\href{mailto:maa5652@gmail.com}{\tt maa5652@gmail.com} \\
\end{center}

\vskip .2 in

\begin{abstract}
Let $S_n$ and $S_{n,k}$ be, respectively, the number of subsets and
$k$-subsets of $\mathbb{N}_n=\{1,\ldots,n\}$ such that no two subset
elements differ by an element of the set $\mathcal{Q}$, the largest
element of which is $q$. We prove a bijection between such $k$-subsets
when $\mathcal{Q}=\{m,2m,\ldots,jm\}$ with $j,m>0$ and permutations
$\pi$ of $\mathbb{N}_{n+jm}$ with $k$ excedances satisfying
$\pi(i)-i\in\{-m,0,jm\}$ for all $i\in\mathbb{N}_{n+jm}$. We also
identify a bijection between another class of restricted permutation
and the cases $\mathcal{Q}=\{1,q\}$ and derive the generating function
for $S_n$ when $q=4,5,6$. We give some classes of
$\mathcal{Q}$ for which $S_n$ is also the number of compositions of
$n+q$ into a given set of allowed parts.
We also prove a bijection between $k$-subsets for a class of
$\mathcal{Q}$ and the set representations of size $k$ of equivalence
classes for the occurrence of a given length-($q+1$) subword within
bit strings. We then formulate a straightforward procedure for
obtaining the generating function for the number of such equivalence
classes.
\end{abstract}

\section{Introduction}
Combinations without specified separations (or \textit{restricted
  combinations}) refer to the subsets of $\Nset_n=\{1,\ldots,n\}$ with
the property that no two elements of a subset differ by an element of
a set $\Qset$, the elements of which are independent of $n$. We denote
the total number of such subsets and the number of $k$-subsets (i.e.,
subsets of size $k$) by $S^\Qset_n$ and $S^\Qset_{n,k}$,
respectively. We drop the superscript when it is clear which set
$\Qset$ of disallowed differences we are dealing with. For any
$\Qset$, one always takes
$S^\Qset_{n<0}=S^\Qset_{n,k<0}=S^\Qset_{n<k,k}=0$.

When there are no restrictions, we just have the number of ordinary
combinations and so $S^{\{\}}_n=2^n$ (\seqnum{A000079} in the OEIS
\cite{Slo-OEIS}) and $S^{\{\}}_{n,k}=\tbinom{n}{k}$
(\seqnum{A007318}).  It appears to be well-known that
$S^{\{1,\ldots,q\}}_n=S^{\{1,\ldots,q\}}_{n-1}+S^{\{1,\ldots,q\}}_{n-q-1}
+\delta_{n+q,0}$, where $\delta_{i,j}$ is 1 if $i=j$ and 0 otherwise,
and $S^{\{1,\ldots,q\}}_{n,k}=\tbinom{n+q(1-k)}{k}$ (\seqnum{A329146}). These reduce to
the classic results $S^{\{1\}}_n=F_{n+2}$, where $F_n$ is the $n$-th
Fibonacci number (\seqnum{A000045},
$F_n=F_{n-1}+F_{n-2}+\delta_{n,1}$, $F_{n<1}=0$), and
$S^{\{1\}}_{n,k}=\tbinom{n+1-k}{k}$ (\seqnum{A011973}) \cite{Kap43}.
Expressions for $S_n^\Qset$ and $S^\Qset_{n,k}$ have also been
obtained when $\Qset=\{q\}$ \cite{Kon81,Pro83,KL91c,KL91b,AE22} and
$\Qset=\{m,2m,\ldots,jm\}$ with $j,m>0$ \cite{MS08,All22}.

In recent work \cite{All-Comb}, which we henceforth refer to as
All24, the author found a bijection between the $k$-subsets of $\Nset_n$ with
disallowed differences set $\Qset$, of which the largest element is $q$,
and the restricted-overlap tilings of an $(n+q)$-board (an
$(n+q)\times1$ board) with squares and $k$ combs. This
enables generating functions for $S^\Qset_n$ and recursion relations
for $S^\Qset_{n,k}$ to be readily obtained for a variety of families
of $\Qset$ \cite{All-Comb}.  A $(w_1,g_1,w_2,\ldots,g_{t-1},w_t)$-comb
is a tile composed of $t>0$ rectangular sub-tiles (called
\textit{teeth}) arranged in such a way that the $i$-th tooth, which
has dimensions $w_i\times1$, is separated from the $(i+1)$-th by a
$g_i\times1$ gap \cite{All-Comb}.  The \textit{length} of a comb is
the sum of the lengths of the teeth and gaps. Each $\Qset$ corresponds
to a length-$(q+1)$ comb $C$ (which we refer to as a
\textit{$\Qset$-comb}) defined as follows. We label the cells of $C$
from 0 to $q$. Cell 0 is always part of the first tooth. Cell $i$
(where $i=1,\ldots,q$) is part of a tooth iff $i\in\Qset$.  Note that
we are free to include 0 in $\Qset$ without affecting the results (and
we do so in Theorems \ref{T:per} and \ref{T:per2} for reasons of
convenience and symmetry) since subset elements are distinct and so
cannot differ by 0. This 0 element can then be regarded as
corresponding to cell 0 of the comb.  To ensure that there is only one
way to specify the $w_i$ and $g_i$ of the $\Qset$-comb, we insist that
all teeth and gaps have positive integer length. A
\textit{restricted-overlap tiling} using squares and $\Qset$-combs
means that any comb can be placed so that its non-leftmost cells (i.e,
cells 1 to $q$) overlap the non-leftmost cells of any other combs on
the board and no other forms of overlap of tiles is allowed
\cite{All-Comb} (see Fig.~\ref{f:twocombmetatiles} for examples).  We
let $B^\Qset_n$ denote the number of restricted-overlap tilings of an
$n$-board with $\Qset$-combs and squares, and $B^\Qset_{n,k}$ denote
the number of such tilings that contain $k$ combs.  The following
theorem is a consequence of the bijection.
\begin{theorem}[Theorem~2 in All24]\label{T:S=B}
For all nonnegative $n$ we have
$S^\Qset_n=B^\Qset_{n+q}$
and $S^\Qset_{n,k}=B^\Qset_{n+q,k}$.
\end{theorem}

Any tiling of an $n$-board can be expressed as a tiling using
\textit{metatiles}, which are gapless groupings of tiles that cannot
be split into smaller gapless groupings by removing a gapless grouping
of tiles \cite{Edw08}. The simplest metatiles when tiling a long
enough board with squares and combs are a single square (denoted by
$S$) and a $(w_1,g_1,\ldots,w_t)$-comb with all the gaps filled by
squares. We call the latter metatile a \textit{filled comb} and it is
denoted by $CS^g$, where $g=\sum_ig_i$. If the comb has one tooth it
is just a $w_1$-omino (and also a metatile by itself), and $S$ and $C$
are then the only metatiles. Otherwise, as a consequence of the
restricted-overlap nature of the tiling, a $C$ can always be placed so
that its leftmost cell occupies the next available empty cell of the
yet-to-be-completed metatile. If more empty cells within the
yet-to-be-completed metatile result (which, if a $C$ is placed
starting in the final cell of the final gap of the first comb on the
board, is the case iff $2w_t<q$ \cite[Lemma 2]{All-Comb}), this
can continue indefinitely and so there are an infinite number of
possible metatiles.

A systematic way to generate all metatiles, and thus obtain generating
functions for $B^\Qset_n$ and $B^\Qset_{n,k}$, is to first construct a
directed pseudograph (henceforth simply referred to as a
\textit{digraph}) in which each node has a label giving a bit-string
representation of the remaining gaps and filled cells (a 0
representing an empty cell, a 1 a filled one) starting at the next gap
in the yet-to-be-completed metatile and ending at the final filled
cell \cite{EA15,All-Comb} (see Figs.~\ref{f:dg14}, \ref{f:dg15}, and
\ref{f:dg16} for examples). The \textit{0~node} represents both the
empty board and the completed metatile. Each arc represents the
addition of a tile and has a label showing the type of tile and a
subscript giving the increase in length of the incomplete metatile
resulting from adding that tile, if that increase is not 0. The length
increase from adding a $C$ is given by $q+1-d$, where $d$ is the
number of binary digits in the node the arc starts from, but with $d=0$
for the 0~node.  A walk starting and ending at the 0~node, but not
visiting it in between, represents a metatile and is denoted by the
string of $C$ and $S$ corresponding to the arcs in the order
visited. For example, in Fig.~\ref{f:dg15}, the loop attached to the
0~node corresponds to the $S$ metatile, and the cycle starting at the
0~node and then visiting the $0^31$ node (an abbreviation for $0001$),
the $0^21$ node, and then the 01 node before returning to the 0~node
corresponds to the $CS^2$ (filled comb) metatile.

A cycle that does not include the 0~node is called an \textit{inner
  cycle} (e.g., $SSC_{[4]}$ in Fig.~\ref{f:dg15}, where including the
length increment subscript quickly identifies which comb arc we are
referring to). There are a finite number of possible metatiles iff the
digraph lacks an inner cycle since a walk can traverse an inner cycle
an arbitrary number of times before returning to the 0~node.  If all
the inner cycles have at least one node in common, that node is called
a \textit{common node} (e.g., the $01^{q-p}$ node is the common node
in the Fig.~\ref{f:dgwba=1} digraph; either the $0^21$ or the 01~node
can be chosen as the common node in the digraph in Fig.~\ref{f:dg14};
Fig.~\ref{f:dg15} and Fig.~\ref{f:dg16} show digraphs that have inner
cycles but no common node). For digraphs with a common node, there are
simple expressions for the recursion relations for the numbers of
tilings $B^\Qset_n$ and $B^\Qset_{n,k}$ in terms of lengths of various
cycles in the digraph \cite{EA15,All-Comb} (we rederive these in the
form of generating functions in \S\ref{s:pcn}).  If the digraph
possesses inner cycles, but no common node, there is a class of digraph
(which has 3 inner cycles) where an analogous simple procedure has
been found for obtaining the recursion relations \cite{All22}.

A systematic search of the OEIS for sequences $(S^\Qset_n)_{n\ge0}$
for all $\Qset$ with $q\leq4$ shows that there are various connections
between some types of restricted combinations and some types of
restricted permutations. One also finds relations to compositions and
bit strings.  Here, in \S\ref{s:srp}, we give bijections between
between two classes of strongly restricted permutations and
combinations with $\Qset=\{1,q\}$ and $\Qset=\{m,2m,\ldots,jm\}$. As a
consequence of the former bijection, we can prove a conjectured
recursion relation for the permutations corresponding to the
$\Qset=\{1,4\}$ case. In \S\ref{s:pcn}, we derive generating functions
for $B^\Qset_n$ and $B^\Qset_{n,k}$ for tilings with metatiles
generated by a particular class of digraph (which includes the 3 inner
cycle class with no common node dealt with previously \cite{All22})
without a common node. This enables us to easily obtain a generating
function for $S^{\{1,5\}}_n$.  In \S\ref{s:comp}, although it is not
all that surprising given the tiling interpretation of restricted
combinations, for completeness, we also show that $S^\Qset_n$ is the
number of compositions of $n+q$ into certain sets of parts for some
classes of $\Qset$.  We point out an elementary connection between
all instances of $\Qset$ and bit strings in \S\ref{s:bits}. This
forms the basis for an algorithm for efficiently calculating
$S^\Qset_n$ and $S^\Qset_{n,k}$ for any given $\Qset$, as detailed via
a heavily annotated C program listed in the Appendix. In
\S\ref{s:bits}, we also give a bijection between $k$-subsets for
certain $\Qset$ and equivalence classes for the occurrence of certain
subwords within bit strings. The bijection enables us to show that a
generating function for the number of equivalence classes is always
straightforward to obtain. In the discussion
(\S\ref{s:dis}), we mention transfer matrices as an
alternative way to obtain the $B^\Qset_n$ and $B^\Qset_{n,k}$
generating functions from digraphs. 

\section{Generating functions for the numbers of walks on
  digraphs}\label{s:pcn}

We start by reviewing further terminology concerning digraphs.  For a
digraph possessing a common node, a path from the 0~node to the common
node followed by a path from the common node back to the 0~node (and
so the common node is only visited once) is called a \textit{common
  circuit} (e.g., for either choice of common node, the digraph in
Fig.~\ref{f:dg14} possesses a single common circuit, namely,
$C_{[5]}S^2$). A cycle that includes the 0~node, but not the common
node, is called an \textit{outer cycle} (e.g., in Fig.~\ref{f:dg14}, the
$S$ loop attached to the 0~node is the only outer cycle).

The proofs in this section require the following three lemmas \cite{Mat14_}.
The bivariate generating functions $G_i(x,y)$ we employ here are such
that the coefficient of $x^ly^k$ of the generating function of a walk
or set of walks between two nodes is the number of configurations of tiles that
contribute $l$ to the overall length of the tiling and contain $k$ combs. 

\begin{figure}[b]
\begin{center}
\includegraphics[width=16cm]{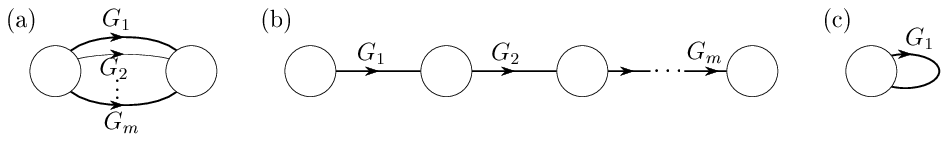}
\end{center}
\caption{Schematic diagrams of connections between nodes corresponding
  to the construction of generating functions via (a)~the parallel
  rule (Lemma~\ref{L:par}), (b)~the series rule (Lemma~\ref{L:ser}),
  and (c)~the loop rule (Lemma~\ref{L:loop}).}
\label{f:gflem}
\end{figure}

\begin{lemma}[The Parallel Rule]\label{L:par}
If two nodes in a digraph are connected via $m$ alternative walks (or
sets of walks), as shown in Fig.~\ref{f:gflem}a, that have associated
generating functions $G_i(x,y)$, where $i=1,\ldots,m$, then the
combined generating function for all of those walks connecting the two
nodes is $\sum_{i=1}^m G_i(x,y)$.
\end{lemma}

\begin{lemma}[The Series Rule]\label{L:ser} 
If two nodes in a digraph are connected via $m-1$ intermediate nodes,
which are connected by walks (or sets of walks), as shown in
Fig.~\ref{f:gflem}b, that have associated generating functions
$G_i(x,y)$, where $i=1,\ldots,m$, then the combined generating
function for all of those walks connecting the two nodes is
$\prod_{i=1}^m G_i(x,y)$.
\end{lemma}

\begin{lemma}[The Loop Rule]\label{L:loop}
If a node is connected to itself by a walk (or set of walks), 
 as shown in Fig.~\ref{f:gflem}c, with a corresponding generating function
$G_1(x,y)$, then the overall generating function corresponding to the
traversal of the loop walk (or walks) zero or more times is
\[
1+G_1(x,y)+G_1^2(x,y)+\cdots=\frac{1}{1-G_1(x,y)}.
\]
\end{lemma}

A form of the following theorem was originally stated by Edwards and Allen
\cite{EA15}, and a more compact version re-derived later
\cite{All-Comb}. Here, using the above three lemmas, we obtain 
the result in terms of a bivariate generating function.
We use this theorem in \S\ref{s:srp1} to obtain
the $S^{\{1,4\}}_n$ generating function.

\begin{theorem}\label{T:cn}
For a digraph possessing a common node, let $\loi$ be the
length of the $i$-th outer cycle ($i=1,\ldots,N\rb{o}$) and
let $\koi$ be the number of combs it contains, let $L_r$ be
the length of the $r$-th inner cycle ($r=1,\ldots,N$) and let $K_r$ be the
number of combs it contains, and let $\lci$ be the length
of the $i$-th common circuit ($i=1,\ldots,N\rb{c}$) and
let $\kci$ be the number of combs it contains. Then the
generating function such that the coefficient of $x^ny^k$ thereof is the
number of tilings an $n$-board that contain $k$ combs is given by
\begin{equation}\label{e:cngf}
G(x,y)=\dsfrac{1-\sum_{r=1}^Nx^{L_r}y^{K_r}}
{1-\sum_{r=1}^Nx^{L_r}y^{K_r}
-\sum_{i=1}^{N\rb{o}}
\biggl(x^\loi y^\koi
-\sum_{r=1}^Nx^{\loi+L_r}y^{\koi+K_r}\biggr)
  -\sum_{i=1}^{N\rb{c}} x^\lci y^\kci}.
\end{equation}
\end{theorem}

\begin{figure}
\begin{center}
\includegraphics[width=16cm]{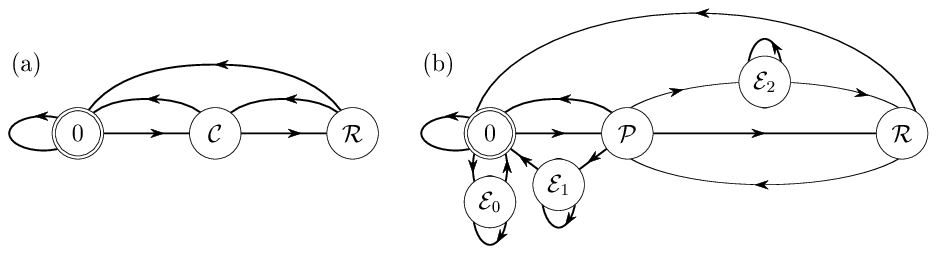}
\end{center}
\caption{Schematic diagrams of digraphs with (a)~a common node $\mathcal{C}$
and (b)~a pseudo-common node $\mathcal{P}$ to which Theorem~\ref{T:cn} and
Theorem~\ref{T:pcn} apply, respectively. In these diagrams, each arc,
with the exception of each errant loop (each arc connecting an
$\mathcal{E}_i$ node to itself), potentially represents a number of
alternative arcs in the specific instance of a digraph. Some nodes and
arcs can
be absent.}
\label{f:cnpcn}
\end{figure}

\begin{proof}
The generating functions corresponding to a single traversal of the
$i$-th outer cycle (represented by the loop attached to the 0 node in
Fig.~\ref{f:cnpcn}a), the $i$-th common circuit (represented by the
path from the 0 node to the $\mathcal{C}$ node and back or the path
$0\to\mathcal{C}\to\mathcal{R}\to0$), and the $r$-th inner cycle
($\mathcal{C}\to\mathcal{R}\to\mathcal{C}$) are, respectively, $x^\loi
y^\koi$, $x^\lci y^\kci$, and $x^{L_r}y^{K_r}$.  The $N$ inner cycles
are alternative paths from the common node back to itself and so, from
the parallel and loop rules, the generating function corresponding to
their traversal (starting from and returning to the common node) in
any order zero or more times is
$1/(1-\sum_{i=1}^Nx^{L_r}y^{K_r})$. Using the series rule, we multiply
this by $x^\lci y^\kci$ to obtain the generating function for the walk
that consists of the paths to and from $\mathcal{C}$
along the $i$-th common circuit and an arbitrary number of
trips around the inner cycles. We now apply the parallel rule to all
such cases and to all outer cycles, followed by the loop rule since
these can be executed any number of times. This gives
\[
G(x,y)=\dsfrac{1}{1-\sum_{i=1}^{N\rb{o}}x^\loi y^\koi
  -\dsfrac{\sum_{i=1}^{N\rb{c}} x^\lci y^\kci}
  {1-\sum_{r=1}^Nx^{L_r}y^{K_r}}},
\]
which yields \eqref{e:cngf} on rearranging.
\end{proof}
  
The following definitions are slightly modified versions of those
given in earlier work on a class of digraph with inner cycles but no
common node in order to incorporate the increased complexity of the
type of digraphs we are about to consider \cite{All22}.  For a digraph
lacking a common node, we refer to any inner cycle that can be
represented as a single arc linking a node to itself as an
\textit{errant loop} if the digraph would only have a common node
$\mathcal{P}$ if all the errant loop arcs were removed. The node
$\mathcal{P}$ is then referred to as a \textit{pseudo-common node}. A
node with an errant loop is called an \textit{errant loop node}.
Evidently, a particular errant loop node $\mathcal{E}$ cannot also be
a pseudo-common node; if that were the case, removing all the other
errant loop arcs except the errant loop at $\mathcal{E}$ would result
in a digraph with a common node, implying that the loop at
$\mathcal{E}$ was never an errant loop in the first place.  For a
digraph with at least one errant loop, a \textit{common circuit} is
now defined as two concatenated simple paths from the 0~node to
$\mathcal{P}$ and from $\mathcal{P}$ to the 0~node. We also need to
modify the definition of an outer cycle: it is now a cycle that
includes the 0~node but not $\mathcal{P}$. Inner cycles (other than
errant loops), outer cycles, and common circuits are said to be
\textit{plain} if they do not contain any errant loop nodes and
\textit{non-plain} otherwise.

\begin{theorem}\label{T:pcn}
If a digraph has plain outer cycles of lengths $\loi$ that contain $\koi$
combs ($i=1\ldots,N\rb{o}$), non-plain outer cycles of lengths $\tloi$
that contain $\tkoi$ combs with associated errant loops having lengths
$\loei$ and $\koei$ combs ($i=1\ldots,\tilde{N}\rb{o}$), 
plain inner cycles (excluding the errant loops) of lengths $L_r$ that
contain $K_r$ combs ($r=1,\ldots,N$),
non-plain inner cycles of lengths $\tilde{L}_r$
that contain $\tilde{K}_r$ combs with associated errant loops having
lengths $\le{r}$ and $\ker$ combs ($r=1,\ldots,\tilde{N}$), 
plain common circuits of lengths
$\lci$ that contain $\kci$ combs ($i=1,\ldots,N\rb{c}$),
non-plain common circuits of lengths
$\tlc{i}$ that contain $\tkci$ combs  with associated errant loops having lengths
$\lce{i}$ and $\kcei$ combs ($i=1,\ldots,\tilde{N}\rb{c}$),
and all non-plain cycles and circuits contain exactly one errant node,
then
the generating function whose coefficient of $x^ny^k$ gives the number
of tilings of an $n$-board that use $k$ combs is given by
\begin{equation}\label{e:pcngf}
G(x,y)=\dsfrac{1}{1-\sum_{i=1}^{N\rb{o}}x^\loi y^\koi
  -\sum_{i=1}^{\tilde{N}\rb{o}}\dsfrac{x^\tloi y^\tkoi}{1-x^\loei y^\koei}
  -\dsfrac{\sum_{i=1}^{N\rb{c}} x^\lci y^\kci
    +\sum_{i=1}^{\tilde{N}\rb{c}}
    \dsfrac{x^{\tlc{i}}y^\tkci}{1-x^\lce{i}y^\kcei}}
  {1-\sum_{r=1}^{N}x^{L_r}y^{K_r}
    -\sum_{r=1}^{\tilde{N}}\dsfrac{x^{\tilde{L}_r}y^{\tilde{K}_r}}
  {1-x^{\le{r}}y^\ker}
}}.
\end{equation}
\end{theorem}

\begin{proof}
The proof is similar to that of Theorem~\ref{T:cn}; we simply have
more possible types of walk.  
The generating functions corresponding to a single traversal of the
$i$-th plain outer cycle, the $i$-th plain common circuit, and the
$r$-th plain inner cycle are, respectively, $x^\loi y^\koi$, $x^\lci
y^\kci$, and $x^{L_r}y^{K_r}$. On application of the series and loop
rules, one sees that the generating functions corresponding to a
single traversal of the $i$-th non-plain outer cycle, the $i$-th
non-plain common circuit, and the $r$-th non-plain inner cycle in
addition to an arbitrary number of traversals of their errant loops
are, respectively, $x^\tloi y^\tkoi/(1-x^\loei y^\koei)$,
$x^{\tlc{i}}y^\tkci/(1-x^{\lce{i}}y^\kcei)$, and
$x^{L_r}y^{K_r}/(1-x^{\le{r}}y^\ker)$. In Fig.~\ref{f:cnpcn}b these
are represented, respectively, by the walks $0\to\mathcal{E}_0\to0$,
$0\to\mathcal{P}\to\mathcal{E}_1\to0$ or
$0\to\mathcal{P}\to\mathcal{E}_2\to\mathcal{R}\to0$, and   
$\mathcal{P}\to\mathcal{E}_2\to\mathcal{R}\to\mathcal{P}$, where,
on reaching $\mathcal{E}_i$, its loop can be executed an arbitrary
number of times. It should now be apparent where the terms in the
denominator of \eqref{e:pcngf} arise from. The first two sums are from
the plain and non-plain outer cycles. The numerator of the final term is the
generating function for all the common circuits (including arbitrary
number of traversals of the errant loops in the case of the non-plain
ones). On reaching $\mathcal{P}$ via the first part of a common
circuit, the inner cycles (and their associated errant loops in the
case of the non-plain ones) can be executed in any order any number of
times, which is what the denominator of the final term corresponds to
(by application of all three rules). 
\end{proof}

To obtain the generating functions for $S^\Qset_{n}$ and
$S^\Qset_{n,k}$, which we denote by $g^{\Qset}(x)$ and
$g^{\Qset}(x,y)$, respectively, we require the following result.

\begin{corollary}\label{C:snkgf}
If $G(x,y)$ is the bivariate generating function for $B^\Qset_{n,k}$
then
\[
g^{\Qset}(x,y)=
\frac1{x^q}\biggl(G(x,y)-\frac{1-x^q}{1-x}\biggr),
\]
where $q$ is the largest element of $\Qset$,
and $g^\Qset(x)=g^\Qset(x,1)$.
\end{corollary}

\begin{proof}
From Theorem~\ref{T:S=B} we have $S^\Qset_{n,k}=B^\Qset_{n+q,k}$,
which implies that we can obtain $g^\Qset(x,y)$ by removing the first
$q$ terms from $G(x,y)$, namely, $1+x+x^2+\cdots+x^{q-1}$ (since the
only way to tile a board of length less than $q$ is with no
$\Qset$-combs), and then by dividing the result by $x^q$ to obtain the
correct offset.      
\end{proof}

\section{Bijections between
  strongly restricted permutations and restricted combinations}
\label{s:srp}

A strongly restricted permutation $\pi$ of the set $\Nset_n$ is a
permutation for which the number of permissible values of $\pi(i)-i$
for each $i\in\Nset_n$ is less than a finite number independent
of $n$~\cite{Leh70}. Here we give two bijections between restricted
combinations and types of
strongly restricted permutations where the permutations satisfy
$\pi(i)-i\in\Dset$.

\subsection{Connection with the class $\Qset=\{1,q\}$}
\label{s:srp1}

Our first bijection concerns a strongly restricted permutation where
$\Dset=\{-1,0,1\}$ and no two of $q$ consecutive $\pi(i)$ may
differ from one another by more than $q$.

\begin{theorem}\label{T:011m}
There is a bijection between the $k$-subsets of $\Nset_n$ such that no
two elements of the subset differ by an element of $\Qset=\{1,q\}$,
where $q=2,3,\ldots$, and the permutations $\pi$ of $\Nset_{n+1}$ that
have $k$ pairs of consecutive items that have been exchanged such 
that for all $i=1,\ldots,n+2-q$ the condition
$\max_{j=0,\ldots,q-1}\pi(i+j)-\min_{j=0,\ldots,q-1}\pi(i+j)\leq q$ is met.
\end{theorem}

\begin{proof}
The permutation $\pi$ of $\Nset_{n+1}$ corresponding to an allowed
subset $\Sset$ of $\Nset_n$ where $\Qset=\{1,q\}$ is formed as
follows. If $i\in\Sset$ (where $i=1,\ldots,n$) then $\pi(i)=i+1$ and
$\pi(i+1)=i$ (i.e., a pair of two consecutive items has been
exchanged) whereas if neither $i$ nor $i-1$ are in $\Sset$ (for
$i=2,\ldots,n+1$) then $\pi(i)=i$. As $\Qset$ contains 1, once a pair
of items has been exchanged, neither item will be moved again.  Given
that $\pi(i)$ can be no more than 1 away from $i$, the only way that
the condition could be violated would be
if $\pi(i+1)=i$ (and hence $\pi(i)=i+1$) and $\pi(i+q)=i+q+1$; this
could only occur if both $i$ and $i+q$ were in $\Sset$. This is
impossible since $q\in\Qset$. The process of forming a permutation
from a subset is clearly reversible and so the bijection is
established.
\end{proof}

For example, the allowed subsets of $\Nset_5$ when $\Qset=\{1,3\}$
followed by the corresponding permutation of $\Nset_6$ are
\{\} 123456, \{1\} 213456, \{2\} 132456, \{3\} 124356, \{4\} 123546,
\{5\} 123465, \{1,3\} 214356, \{2,4\} 132546, \{3,5\} 124365, \{1,5\}
213465, and \{1,3,5\} 214365.

Theorem~\ref{T:011m} has connections to the following sequences in the
OEIS \cite{Slo-OEIS}: \seqnum{A000930} ($S^{\{1,2\}}_n$) and
\seqnum{A102547} ($S^{\{1,2\}}_{n,k}$), its associated triangle;
\seqnum{A130137} ($S^{\{1,3\}}_n$), which we mention again in the
section on bit strings; \seqnum{A263710} ($S^{\{1,4\}}_n$), which was
the sequence that led to the discovery of the bijection and we now
prove the conjectured generating function for that entry; along with
\seqnum{A374737} ($S^{\{1,5\}}_n$) and \seqnum{A385870}
($S^{\{1,6\}}_n$), which the author recently added to the OEIS.

The case $\Qset=\{1,4\}$ falls into the class covered by Theorem~8 in
All24, which gives recursion relations for
$B^{\Qset}_n$ and $B^{\Qset}_{n,k}$.  However, in order to avoid
introducing the notation used in that theorem at this stage, we
instead obtain the generating function using Theorem~\ref{T:cn}, which
is also straightforward and illustrative to do.

\begin{figure}
\begin{center}
\includegraphics[width=7cm]{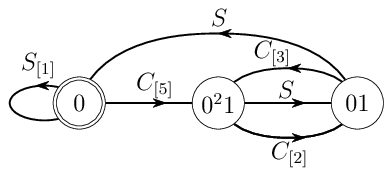}
\end{center}
\caption{Digraph for restricted-overlap tiling with squares ($S$) and
  $(2,2,1)$-combs ($C$), which corresponds to the $\Qset=\{1,4\}$
  case. In this and subsequent figures showing digraphs, each number
  in square brackets gives the contribution to the length of the
  metatile resulting from the addition of the tile. The subscript is
  omitted if there is no contribution, as is the case for any square
  placed in a gap of a comb.}
\label{f:dg14}
\end{figure}

\begin{proposition}
The generating function $g^{\{1,4\}}(x)$ for the number of restricted
combinations of $\Nset_n$ with disallowed differences $\{1,4\}$ is
given by
\begin{equation}\label{e:g14}
g^{\{1,4\}}(x)=\frac{1+x+x^2+x^3+2x^4+x^6}{1-x-x^3+x^4-2x^5+x^6-x^7}.
\end{equation}
\end{proposition}  
\begin{proof}
The $\Qset$-comb is a $(2,2,1)$-comb.  The metatile-generating digraph
for tiling with squares and such combs (Fig.~\ref{f:dg14}) has two
inner cycles ($SC_{[3]}$ and $C_{[2]}C_{[3]}$). These have lengths
$L_1=3$ and $L_2=5$.  We choose $0^21$ as the common node. The common
circuits ($C_{[5]}S^2$ and $C_{[5]}C_{[2]}S$) have lengths
$l\rb{c1}=5$ and $l\rb{c2}=7$.  There is a single outer cycle ($S$).
This has length $l\rb{o1}=1$. Then from Theorem~\ref{T:cn}
and Corollary~\ref{C:snkgf} we get
\begin{equation}\label{e:g14w}
g^{\{1,4\}}(x)=
\frac1{x^4}\biggl(
\frac{1-x^3-x^5}{1-x^3-x^5-x+x^4+x^6-x^5-x^7}-\frac{1-x^4}{1-x}\biggr),
\end{equation}
which reduces to \eqref{e:g14}.
\end{proof}  

As our first application of Theorem~\ref{T:pcn}, we also obtain the
generating function for the next sequence in the series, i.e., that
corresponding to $\Qset=\{1,5\}$ (\seqnum{A374737}).

\begin{proposition}
The generating function $g^{\{1,5\}}(x)$ for the number of restricted
combinations of $\Nset_n$ with disallowed differences $\{1,5\}$ is
given by
\begin{equation}\label{e:g15}
  g^{\{1,5\}}(x)
  =\frac{1+2x+2x^2+2x^3+2x^4+4x^5+3x^6+2x^7+x^8+x^9}
  {1-x^2-x^3-x^4+x^5-x^6-x^7-x^8-x^{10}}.
\end{equation}
\end{proposition}  

\begin{figure}
\begin{center}
\includegraphics[width=10.5cm]{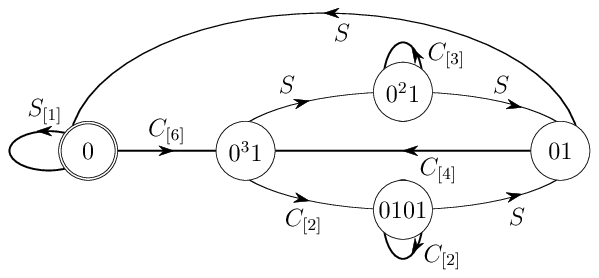}
\end{center}
\caption{Digraph for restricted-overlap tiling with squares ($S$) and
  $(2,3,1)$-combs ($C$), which corresponds to the $\Qset=\{1,5\}$ case.}
\label{f:dg15}
\end{figure}

\begin{proof}
In this case, the $\Qset$-comb is a $(2,3,1)$-comb and the
metatile-generating digraph (Fig.~\ref{f:dg15}) falls into the
category covered by Theorem~\ref{T:pcn}. We choose the $0^31$ node as the
pseudo-common node $\mathcal{P}$.
 There
is just one outer cycle ($S$). It is plain and its length $l\rb{o1}=1$.
The common circuits ($C_{[6]}S^3$ and
$C_{[6]}C_{[2]}S^2$) are non-plain and
have lengths $\tlc{1}=6$ and $\tlc{2}=8$;
their respective errant loops ($C_{[3]}$ and $C_{[2]}$) have
lengths $\lce{1}=3$ and $\lce{2}=2$.
The inner cycles passing through $\mathcal{P}$
($S^2C_{[4]}$ and $C_{[2]}SC_{[4]}$) are non-plain and
have lengths $\tilde{L}_1=4$ and $\tilde{L}_2=6$;
their respective errant loops ($C_{[3]}$ and $C_{[2]}$) have
lengths $\le{1}=3$ and $\le{2}=2$. 
From \eqref{e:pcngf} and Corollary~\ref{C:snkgf} we have
\[
g^{\{1,5\}}(x)=\frac1{x^5}\left(
\dsfrac1{1-x-\dsfrac{\dsfrac{x^6}{1-x^3}+\dsfrac{x^8}{1-x^2}}
  {1-\dsfrac{x^4}{1-x^3}-\dsfrac{x^6}{1-x^2}}}
  -\frac{1-x^5}{1-x}
\right),
\]
which reduces to \eqref{e:g15}.
\end{proof}

The metatile-generating digraph corresponding to the next case in the
$\{1,q\}$ series (Fig.~\ref{f:dg16}) does not fall into any of the
classes of digraph that we have derived a generating function for so
far. However, as we now show, the systematic application of the
digraph generating function rules (Lemmas \ref{L:par}, \ref{L:ser},
and \ref{L:loop}) can still be used to yield the generating function
in this instance. To assist in breaking the derivation into more
manageable chunks, we introduce the following notation. We let
$\Gpath{X}{Y}{Z}$ denote the generating function corresponding to all
possible walks from node $X$ to node $Y$ that do not include node $X$
a second time and never encounter node $Z$.  We let $\Gloop{X}{Z}$
denote the generating function corresponding to all possible walks
from node $X$ to itself without visiting itself or node $Z$ in
between.

\begin{proposition}
The generating function $g^{\{1,6\}}(x)$ for the number of restricted
combinations of $\Nset_n$ with disallowed differences $\{1,6\}$ is
given by
\zmlg
\begin{multline}\label{e:g16}
  g^{\{1,6\}}(x)=\\
\frac{1{\qp}x{\qp}x^2{\qp}2x^3{\qp}2x^4{\qp}2x^5{\qp}5x^6{\qp}3x^8{\qp}2x^9{\qm}x^{10}{\qp}x^{11}{\qm} 
 4x^{12}{\qm}x^{13}{\qm}2x^{14}{\qm}2x^{15}{\qm}x^{17}}{1{\qm}x{\qm}x^4{\qm}x^5{\qp}2x^6{\qm}4x^7{\qp} 
  2x^8{\qm}2x^{10}{\qp}2x^{11}{\qm}4x^{12}{\qp}3x^{13}{\qm}x^{14}{\qp}2x^{16}{\qm}x^{17}{\qp}x^{18}}.
\end{multline}
\rmlg
\end{proposition}  

\begin{figure}
\begin{center}
\includegraphics[width=10.5cm]{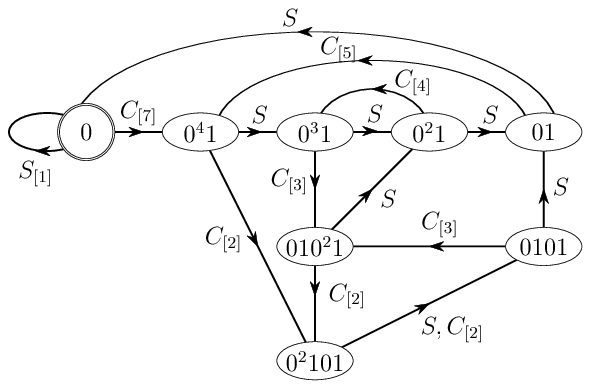}
\end{center}
\caption{Digraph for restricted-overlap tiling with squares ($S$) and
  $(2,4,1)$-combs ($C$), which corresponds to the $\Qset=\{1,6\}$
  case. A comma-separated list of labels at an arc signifies a
  multiarc with those labels.}
\label{f:dg16}
\end{figure}

\begin{proof}
The generating function $G(x,y)$ corresponding to walks starting at
the 0 node of the digraph in Figure~\ref{f:dg16} satisfies
\[
\frac1{G(x,y)}=1-x-\frac{x^7y\Gpath{0^41}{01}{0}}{1-\Gloop{0^41}{0}},
\]
where, owing to the fact that the only way to return to the $0^41$
node is via the $C_{[5]}$ arc from the 01 node, we have
$\Gloop{0^41}{0}=x^5y\Gpath{0^41}{01}{0}$. The arcs leaving the $0^41$
node are $S$ and $C_{[2]}$ (the latter contributes a factor of $x^2y$
to the corresponding term in the generating function) and the only way
to reach the 01 node is from the $0^21$ and 0101 nodes. Hence
\[
\Gpath{0^41}{01}{0}=\frac{\Gpath{0^31}{0^21}{01}+\Gpath{0^31}{0101}{01}}{1-\Gloop{0^31}{01}}
+\frac{x^2y\bigl(\Gpath{0^2101}{0^21}{01}+\Gpath{0^2101}{0101}{01}\bigr)}{1-\Gloop{0^2101}{01}}.
\]
There is a direct route from $0^31$ to $0^21$ along with those that
leave $0^31$ via the $C_{[3]}$ arc (which contributes a factor of
$x^3y$). On reaching $010^21$ there are two possible loops it can
execute ($C_{[2]}SC_{[3]}$ or $C_{[2]}^2C_{[3]}$) before reaching
$0^21$. Hence 
\[
\Gpath{0^31}{0^21}{01}=1+\frac{x^3y}{1-x^5y^2-x^7y^3}.
\]
Since the only allowed route returning to the $0^31$ node is the
$C_{[4]}$ arc from the $0^21$ node,
$\Gloop{0^31}{01}=x^4y \Gpath{0^31}{0^21}{01}$. In a similar fashion
we obtain
\[
\Gpath{0^31}{0101}{01}=\frac{x^5y^2(1+x^2y)}{1-x^5y^2-x^7y^3}, \qquad
\Gpath{0^2101}{0^21}{01}=\frac{(1+x^2y)x^3y}{1-x^4y-x^7y^2}, 
\]
and $\Gpath{0^2101}{0^21}{01}=1+x^2y$. The $1-x^7y^2/(1-x^4y)$ in
\[
\Gloop{0^2101}{01}=\dsfrac{(1+x^2y)x^5y^2}{1-\frac{x^7y^2}{1-x^4y}}
\]
results from a loop within a loop: the $SC_{[4]}$ loop can be executed
on arriving at the $0^31$ node while traversing the $SC_{[4]}C_{[3]}$
loop.  Combining the above results to give $G(x,y)$ and applying
Corollary~\ref{C:snkgf} yields \eqref{e:g16} after simplifying.
\end{proof}

\subsection{Connection with the class $\Qset=\{m,2m,\ldots,jm\}$}
\label{s:srp2}

The proof we give of the second bijection (Theorem~\ref{T:bij}) requires the
following bijection (established in previous work \cite{EA15}) between strongly
restricted permutations and tiling with $(\frac12,g)$-fences (which
are just $(\frac12,g,\frac12)$-combs), where $g$ is always a
nonnegative integer. As the teeth (originally called posts in
the context of fences) are not of integer width, it is convenient to
regard each cell as being divided into two \textit{slots} into
which a tooth can fit \cite{EA20a}.

\begin{lemma}\label{L:srpbij}
There is a bijection between (i) the permutations of $\Nset_n$ satisfying
$\pi(i)-i\in \mathcal{D}$ for each $i\in\Nset_n$ and (ii) the tilings of
an $n$-board using $(\frac12,d_j)$-fences with the left tooth always
placed in a left slot and $d_j$ equal to the non-negative
elements of $\mathcal{D}$ and $(\frac12,-d_k-1)$-fences with their
gaps aligned with cell boundaries and $d_k$ equal to the negative
elements of $\mathcal{D}$. 
If the 
left slot of cell $i$ of the $n$-board is 
occupied by the left (right) tooth of a $(\frac12,g)$-fence then
$\pi(i)=i+g$ ($\pi(i)=i-g-1$).
\end{lemma}
For example,
the tiling in
Fig.~\ref{f:combfromfences} corresponds to the permutation
\[
\biggl(\!\!\begin{array}{ccccccccccccc}1&2&3&4&5&6&7&8&9&10&11&12&13\\
  1&11&3&4&2&6&7&5&9&10&8&12&13
\end{array}\!\!\biggr).
\]
We refer to a fence corresponding to an excedance (i.e., a position $i$
such that $\pi(i)>i$) as an \textit{up fence} and a fence 
giving $\pi(i)-i<0$ as a
\textit{down fence}. The case of a fixed point ($\pi(i)=i$)
corresponds to a gapless fence. This is just an ordinary square tile
(that must be aligned with a cell). 

The following bijection is an extension of that noted but not proved
by Balti\'c \cite{Bal12}. The proof we give requires the notion of a
comb with $t$ teeth all of length $w$ separated by gaps of length $g$, which we
refer to as a $(w,g;t)$-comb \cite{AE23,AE24}.

\begin{theorem}\label{T:bij}
There is a bijection between the $k$-subsets of $\Nset_n$ with
disallowed differences
$\Qset=\{m,2m,\ldots,jm\}$, where $j,m\geq 1$, and the permutations
$\pi$ of $\Nset_{n+jm}$ satisfying $\pi(i)-i\in\{-m,0,jm\}$ for all
$i\in\Nset_{n+jm}$ that contain $k$ excedances.
\end{theorem}

\begin{figure}
\begin{center}
\includegraphics[width=11.5cm]{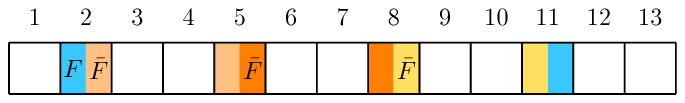}
\end{center}
\caption{A 13-board tiled with an up $(\frac12,9)$-fence ($F$; blue),
  three down $(\frac12,2)$-fences ($\down{F}$; yellow, orange, and
  ochre), and 9~squares.}
\label{f:combfromfences}
\end{figure}

\begin{proof}
From Lemma~\ref{L:srpbij}, the number of such permutations is the
number of ways to tile an $(n+jm)$-board using squares aligned with
the cell boundaries ($S$), up $(\frac12,jm)$-fences ($F$), and down
$(\frac12,m-1)$-fences ($\down{F}$).  Suppose the left tooth of the
first $F$ on the board occupies the left slot of cell $i$. The only
possible way to fill the right slot of cell $i$ is with the left tooth
of an $\down{F}$. The right tooth of the $\down{F}$ lies on the left
slot of cell $i+m$. If $j=1$, the right slot of this cell is occupied
by the right tooth of the first $F$. Otherwise it must again be
occupied by the left tooth of another $\down{F}$ and similarly until
the right tooth of the first $F$ is reached at cell $i+jm$. Hence this
$F$ has $j$ $\down{F}$ placed end-to-end in its interior. Combined,
these $j+1$ tiles occupy exactly the same cells as a
$(1,m-1;j+1)$-comb (Fig.~\ref{f:combfromfences} shows an example when
$j=m=3$). All subsequent $F$ on the board behave similarly. Note that
owing to the fact that any $\down{F}$ must have its left tooth in the
right slot of a cell, such tiles can only appear on the board in
conjunction with $j-1$ other $\down{F}$ surrounded by an $F$.  Hence
there is a bijection between tilings of an $(n+jm)$-board using $k$
$F$, $jk$ $\down{F}$ and $n+jm-(j+1)k$ $S$ and the tilings of a board
of the same length using $k$ $(1,m-1;j+1)$-combs and
$n+jm-(j+1)k$ squares.
\end{proof}

We let $P^\Dset_n$ denote the number of permutations $\pi$ of
$\Nset_n$ such that $\pi(i)-i\in\Dset$ for all $i\in\Nset_n$ and
$P^\Dset_{n,k}$ denote the number of such permutations with $k$
excedances.  We now use the bijection and previously obtained results
to relate $P^{\{-m,0,jm\}}_n$ and $P^{\{-m,0,jm\}}_{n,k}$ to some
generalized Fibonacci numbers and the coefficients of their
corresponding polynomials, respectively.  We let $f^{(t)}_n$ denote
the $n$-th $(1,t)$-bonacci number defined by
$f^{(t)}_n=f^{(t)}_{n-1}+f^{(t)}_{n-t}+\delta_{n,0}$,
$f^{(t)}_{n<0}=0$, and $f^{(t)}_{n}(x)$ denote the $(1,t)$-bonacci
polynomial defined by
$f^{(t)}_n(x)=f^{(t)}_{n-1}(x)+xf^{(t)}_{n-t}(x)+\delta_{n,0}$,
$f^{(t)}_{n<0}(x)=0$.

\begin{corollary}\label{C:bij}
For $m,j\ge1$, $i,k\ge0$, and $r=0,\ldots,m-1$, we have
\begin{align}\label{e:Pn}
P^{\{-m,0,jm\}}_{im+r}&=
\bigl(f^{(j+1)}_i\bigr)^{m-r}\bigl(f^{(j+1)}_{i+1}\bigr)^r,\\
P^{\{-m,0,jm\}}_{im+r,k}&=[x^k]
\biggl(\bigl(f^{(j+1)}_i(x)\bigr)^{m-r}\bigl(f^{(j+1)}_{i+1}(x)\bigr)^r\biggr),
\label{e:Pnk}
\end{align}
where $[x^k]R(x)$ denotes the coefficient of $x^k$ in $R(x)$.
\end{corollary}
\begin{proof}
Using a previously obtained result \cite[Corollary~11]{All22}, for $l\ge0$, $m\ge1$, $t\ge2$, and
$r=0,\ldots,m-1$, we have
\begin{equation}\label{e:Sn=ff}
S^{\{m,2m,\ldots,(t-1)m\}}_{lm+r}=
\bigl(f^{(t)}_{l+t-1}\bigr)^{m-r}\bigl(f^{(t)}_{l+t}\bigr)^r.
\end{equation}
As a consequence of the bijection given in Theorem~\ref{T:bij}, 
$P^{\{-m,0,jm\}}_{n+jm}=S^{\{m,2m,\ldots,jm\}}_n$ and so
$P^{\{-m,0,jm\}}_{n}=S^{\{m,2m,\ldots,jm\}}_{n-jm}$.
Replacing $n$ by $im+r$ in this and using
\eqref{e:Sn=ff} with $t$ replaced by $j+1$ and $l$ replaced by $i-t+1$
gives \eqref{e:Pn}. Similarly, using another result from previous work
\cite[Corollary~10]{All22}, we obtain
\[
S^{\{m,2m,\ldots,(t-1)m\}}_{lm+r,k}=[x^k]
\biggl(\bigl(f^{(t)}_{l+t-1}(x)\bigr)^{m-r}\bigl(f^{(t)}_{l+t}(x)\bigr)^r
\biggr).
\]
Combining this with the result from Theorem~\ref{T:bij} that
$P^{\{-m,0,jm\}}_{n+jm,k}=S^{\{m,2m,\ldots,jm\}}_{n,k}$ in the same
way gives \eqref{e:Pnk}.
\end{proof}  

Sequences in the OEIS that relate to Theorem~\ref{T:bij} and
Corollary~\ref{C:bij} along with the
corresponding $\mathcal{D}$ are as follows: \seqnum{A000045}
$\{-1,0,1\}$, \seqnum{A006498} $\{-2,0,2\}$,
\seqnum{A006500} $\{-3,0,3\}$, 
\seqnum{A031923} $\{-4,0,4\}$,
\seqnum{A224809} $\{-2,0,4\}$,
\seqnum{A224810} $\{-3,0,6\}$,
\seqnum{A224811} $\{-2,0,8\}$,
\seqnum{A224812} $\{-2,0,10\}$,
\seqnum{A224813} $\{-2,0,12\}$,
\seqnum{A224814} $\{-3,0,9\}$,
\seqnum{A224815} $\{-4,0,8\}$.

\section{Connections with compositions}
\label{s:comp}

A composition of a positive integer $n$ is a way of expressing $n$ as
a sum of positive integers (referred to as \textit{parts}) where the
order of the parts is significant; thus, for example, 1+2 and 2+1 are
different compositions of 3. For convenience, it is customary to take
the number of compositions of 0 (any negative integer) as being 1 (0).
It is well known that the number of compositions of $n$ is the same as
the number of tilings of an $n$-board using $m$-ominoes for arbitrary
nonnegative integers $m$ and hence that the number of compositions of
$n$ into distinct parts $m_1,m_2,\ldots,m_i,\ldots$ is given by
$B_n=\delta_{n,0}+\sum_{i=1}B_{n-m_i}$, $B_{n<0}=0$ for all integers
$n$. It follows that any tiling in which all possible metatiles have
distinct lengths corresponds to a composition into parts equal to the
lengths of the metatiles.

\subsection{Compositions with parts drawn from finite sets}

We first examine two classes of $\Qset$ where there are a finite
number of possible metatiles (of distinct lengths) when
restricted-overlap tiling with squares and $\Qset$-combs. Recall that
the number of possible metatiles is finite iff $2r\geq q$, where $r$
is the length of the rightmost tooth of the $\Qset$-comb
\cite[Lemma~2]{All-Comb}. It is sometimes more convenient to express
this condition as $r\geq s-1$, where $s=q+1-r$ is the sum of the
lengths of the teeth and gaps up to the rightmost tooth.

\begin{lemma}\label{L:k<=2}
Suppose $q>0$ is the largest element of $\Qset$. Then $S^\Qset_{n}$ is
the number of compositions of $n+q$ into parts which are drawn from a
finite set of distinct parts iff every metatile, when
restricted-overlap tiling with squares and $\Qset$-combs, contains no
more than two combs. Then the parts are 1, $q+1$, and $q+1+p_i$ (for $i=1,\ldots,\abs{\Nset_q\setminus\Qset}$), where
$p_i$ is the cell number of the $i$-th empty cell in the $\Qset$-comb.
\end{lemma}  
\begin{proof}
If there are a finite number of parts, there are a finite number of
metatiles. This means that, in all cases, the start of the final comb
in a metatile containing more than one comb must lie within the first
comb (see the proof of Lemma~2 in All24). If there were a
third comb present, the cells it occupies without overlapping (of
which there must be at least one, i.e., its cell 0) could be replaced
by squares and give rise to a metatile of the same length. The parts
would then not be distinct. The square and filled comb metatiles are
of length 1 and $q+1$, respectively. The remaining metatiles contain
two combs. If the start of the second comb is in cell $p_i$ of the
first comb, the resulting metatile (completed by filling any remaining
empty cells with squares) is easily seen to have a length of
$q+1+p_i$.
\end{proof}  

In the proofs of the following two theorems we employ combs that, in
general, have a periodic pattern of teeth and gaps before the
rightmost tooth is reached. Extending the notation used for combs with
teeth all the same length and gaps all the same length
\cite{AE23,AE24}, we call an $(l,g,l,g,\ldots,r)$-comb with $t$
teeth an $(l,g,r;t)$-comb and we refer to an
$(l,g,m,h,l,g,m,h,\ldots,r)$-comb with $t$ teeth as an
$(l,g,m,h,r;t)$-comb. Note that an $r$-omino (an $(l,g,r)$-comb) can
be regarded as an $(l,g,r;t)$- or $(l,g,m,h,r;t)$-comb with
$t=1$ ($t=2$).  
Recall that the inclusion of 0 in $\Qset$ is merely for convenience in the
statement of the theorems and is to be removed when giving particular
instances of $\Qset$ that the theorems apply to.

\begin{theorem}\label{T:per}
If
\begin{multline*}
\Qset=\{0,1,\ldots,l-1,l+g,l+g+1,\ldots,2l+g-1,2(l+g),2(l+g)+1,\ldots,\\
(t-1)(l+g),\ldots, (t-1)(l+g)+r-1\equiv q\}
\end{multline*}
for some $l\ge g>0$, $t\in\Zset^+$, $q>0$, and, for $t>1$, $r\ge
(t-1)(l+g)-1$, then $S_n^\Qset$ is the number of compositions of $n+q$
into parts 1 and $q+1$ when $t=1$ and parts 1, $q+1$,
$q+l+1,\ldots,q+l+g$,
$q+2l+g+1,\ldots,q+2(l+g),\ldots,q+(t-1)l+(t-2)g+1,\ldots,q+(t-1)(l+g)$
otherwise.  Note that $\Qset$ reduces to $\{0,1,\ldots,r-1=q\}$ when
$t=1$.
\end{theorem}

\begin{figure}
\begin{center}
\includegraphics[width=10cm]{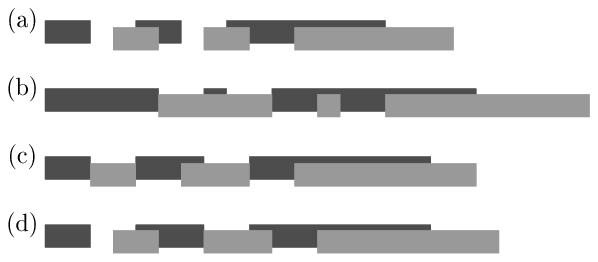}
\end{center}
\caption{Construction of instances of metatiles containing two combs
  when the combs are (a)~(2,2,7;3)-combs, (b)~(5,2,1,2,9;3)-combs,
  and (c,d)~(2,2,3,2,8;3)-combs. In each case, the second comb is
  displaced downwards a little from its final position so that
  parts of the cells of the first comb it overlaps are visible. 
}
\label{f:twocombmetatiles}
\end{figure}

\begin{proof}
The $\Qset$-comb in this case is an $(l,g,r;t)$-comb.  When $t=1$,
there are no two-comb metatiles; the condition $r>1$ ensures that the
filled comb metatile (which in this case is just the comb) is not the
same length as the square. For $t>1$, the condition $r\ge
(t-1)(l+g)-1$ is the same as $2r\geq q$ and thus establishes that
there are a finite number of metatiles. Then the condition $l\ge g$
ensures that a metatile can have at most two combs
(Fig.~\ref{f:twocombmetatiles}(a)). This is because if $g>l$, the
leftmost tooth of a comb and the leftmost cell of another comb could
both be placed within the first gap of the initial comb. The result
then follows from Lemma~\ref{L:k<=2}.
\end{proof}  

OEIS sequences corresponding to a $\Qset$ that Theorem~\ref{T:per}
applies to along with the set itself are as follows:
\seqnum{A000045} \{1\}, 
\seqnum{A006498} \{2\},
\seqnum{A000930} \{1,2\},
\seqnum{A079972} \{2,3\}, 
\seqnum{A003269} \{1,2,3\}, 
\seqnum{A351874} \{1,3,4\},
\seqnum{A121832} \{2,3,4\},
\seqnum{A003520} \{1,2,3,4\}, 
\seqnum{A375985} \{1,3,4,5\}, 
\seqnum{A259278} \{2,3,4,5\},
\seqnum{A005708} \{1,2,3,4,5\},
\seqnum{A276106} \{2,3,4,5,6\}, 
\seqnum{A005709} \{1,2,3,4,5,6\},
\seqnum{A322405} \{2,3,4,5,6,7\},
\seqnum{A005710} \{1,2,3,4,5,6,7\},
\seqnum{A368244} \{2,3,4,5,6,7,8\}.

Theorem~\ref{T:per} covers all cases where there are a finite number
of metatiles for $q\leq7$.  The following generalization of the theorem
(in which the $\Qset$-combs are periodic in the first two teeth and
gaps up to the final tooth) extends this to $q\leq11$.

\begin{theorem}\label{T:per2}
If
\begin{align*}
  \Qset&=\{0,1,\ldots,l-1,l+g,l+g+1,\ldots,l+g+m-1,\\&\qquad
  l+g+m+h, l+g+m+h+1, \ldots, 2l+g+m+h-1,\\&\qquad
  2(l+g)+m+h,2(l+g)+m+h+1,\ldots,2(l+g+m)+h-1,\ldots,*\},\\
  * &=\begin{cases}
  \dfrac{t-1}{2}(l+g+m+h),
  \ldots,
  \dfrac{t-1}{2}(l+g+m+h)+r-1=q>0, & \text{t odd};\\
  \dfrac{t}{2}(l+g)+\Bigl(\dfrac{t}{2}-1\Bigr)(m+h), 
  \ldots,
  \dfrac{t}{2}(l+g)+\Bigl(\dfrac{t}{2}-1\Bigr)(m+h)+r-1=q, & \text{t even};
  \end{cases}
\end{align*}
where $g,m,h,t\in\Zset^+$, and, for $t>1$, $r\ge
\frac12(t-1)(l+g+m+h)-1$ if $t$ is odd and $r\ge
\frac12t(l+g)+(\frac12t-1)(m+h)-1$ if $t$ is even, and either
(i)~$l\ge g+m+h$ or (ii)~$l\ge g,h$ and $m\ge l-1+g$, then $S_n^\Qset$
is the number of compositions of $n+q$ into parts 1 and $q+1$ when
$t=1$ and, when $t>1$, parts 1, $q+1$, $q+l+1,\ldots,q+l+g$,
$q+l+g+m+1,\ldots,q+l+g+m+h,\ldots,*$, where
\[
*=\begin{cases}
  q+\dfrac{t-1}{2}(l+g+m)+\dfrac{t-3}{2}h+1,
  \ldots,
  q+\dfrac{t-1}{2}(l+g+m+h), & \text{t odd};\\
  q+\dfrac{t}{2}l+\Bigl(\dfrac{t}{2}-1\Bigr)(g+m+h)+1, 
  \ldots,
  q+\dfrac{t}{2}(l+g)+\Bigl(\dfrac{t}{2}-1\Bigr)(m+h), & \text{t even}.
\end{cases}
\]
\end{theorem}

\begin{proof}
Here the $\Qset$-comb is an $(l,g,m,h,r;t)$-comb. The $t=1,2$ cases
are covered by the proof of Theorem~\ref{T:per}. First note that if
$t$ is odd, the length of the comb $q+1=\frac12(t-1)(l+g+m+h)+r$ and
it is $\frac12t(l+g)+(\frac12t-1)(m+h)+r$ if $t$ is even from which
the respective conditions on $r$ (so that there are a finite number of
metatiles) follow. There are two mutually exclusive sets of conditions
that ensure that the metatiles contain no more than two
combs. Condition (i) that $l\ge g+m+h$ (see
Fig.~\ref{f:twocombmetatiles}(b)) arises as a result of the
observation that if $l$ were one less than $g+m+h$ and a second comb
were placed starting at the first gap of the first comb, the start a
third comb could be placed starting at the final cell of the second
gap of the first comb. The first part of condition (ii), namely that
$l\ge g,h$ is for the same reasons as the $l\ge g$ condition in
Theorem~\ref{T:per}. We must also have $l+g+m\geq g+m+h$ (see
Fig.~\ref{f:twocombmetatiles}(c)). This also implies that $l\ge
h$. Were $l+g+m$ one less than $g+m+h$, it can be seen that after
placing the second comb at the start of the first gap, a third comb
could be placed at the end of the second gap of the first comb. At the
same time, we must ensure that $m$ is not too small; otherwise the
start of a third comb could be inserted at the start of the second gap
of the first comb. Thus $l-1+g\leq m$ (see
Fig.~\ref{f:twocombmetatiles}(d)).  The result then follows from
Lemma~\ref{L:k<=2}.
\end{proof}  

\subsection{Compositions with parts drawn from infinite sets}

We now turn to compositions where the parts are drawn from infinite sets.
In the context of tiling this means that there is an infinite number
of possible metatiles and so the metatile-generating digraph
has an inner cycle. We start by showing that the metatile lengths can only
be all different if there is just one inner cycle.

\begin{lemma}\label{L:I1I2}
If the metatile-generating digraph for a tiling contains a common node
and more than one inner cycle then there exist distinct metatiles of
the same length.
\end{lemma}
\begin{proof}
Suppose the digraph contains inner cycles $I_1$ and $I_2$. If the
sequence of nodes and arcs to create one metatile involves executing
$I_1$ followed by $I_2$ then a distinct metatile of the same length
can be generated by swapping the order in which the inner cycles are
traversed.
\end{proof}  

Four classes of $\Qset$ whose associated digraphs possess common nodes
have been investigated 
\cite[Theorems 5--8]{All-Comb}.
One of these classes (Theorem~8 in All24)
always has at least two inner cycles; the simplest member of this
class is $\Qset=\{1,4\}$ (see Fig.~\ref{f:dg14}). Another class
(Theorem~7 in All24) has just one inner cycle but always has
the metatiles $C_{[q+1]}C_{[l]}^2$ and $C_{[q+1]}S^{a-1}C_{[2l]}S$,
where $a=\abs{\Nset_q\setminus\Qset}$, which are of the same length (see
Fig.~5 in All24). The metatile-generating digraph of one of
the remaining classes has $a$ inner cycles (Theorem~5 in
All24). The $a=1$ case leads to the following result.

\begin{theorem}\label{T:compwb}
If $\Qset=\{1,\ldots,p-1,p+1,\ldots,q\}$, where $2p>q$, then
$S^\Qset_n$ is the number of compositions of $n+q$ into parts 1 and
$q+1+jp$ for $j=0,1,2,\ldots$. 
\end{theorem}

\begin{figure}
\begin{center}
\includegraphics[width=7cm]{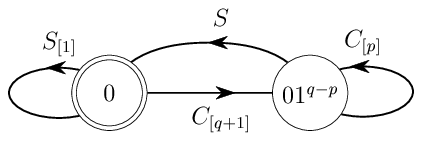}
\end{center}
\caption{Digraph for restricted-overlap tiling with squares ($S$) and
  $(p,1,q-p)$-combs ($C$), where $2p>q$. This corresponds to the
  $\Qset=\{1,\ldots,p-1,p+1,\ldots,q\}$ case.}
\label{f:dgwba=1}
\end{figure}

\begin{proof}
From the digraph (Fig.~\ref{f:dgwba=1}) it can be seen that the
metatiles are $S$ and $C_{[q+1]}C_{[p]}^jS$ for $j=0,1,2,\ldots$. These
are of lengths 1 and $q+1+jp$, respectively.
\end{proof}

OEIS sequences corresponding to a $\Qset$ that Theorem~\ref{T:compwb}
applies to along with the set itself are as follows:
\seqnum{A130137} \{1,3\},
\seqnum{A317669} \{1,2,4\},
\seqnum{A375185} \{1,2,3,5\},
\seqnum{A375186} \{1,2,4,5\}.

The following theorem originates from Theorem~6 in All24 and as it is
a rather general theorem, it employs some extra notation that we now
describe. We let $\theta$ be the bit string representation of $\Qset$
whereby the $j$-th bit from the right in $\theta$ is 1 if and only if
$j\in\Qset$ (we use the same bit string representation of $\Qset$ in
the program given in the Appendix). By $\floor{\theta/2^b}$ we mean
discarding the rightmost $b$ bits in $\theta$ and shifting the
remaining bits to the right $b$ places (we employ the same notation in Theorem~\ref{T:bs}). We use $\mid$ to denote the
bitwise OR operation.

\begin{theorem}\label{T:compmin1arc}
Let $p_i\in\Nset_q\setminus\Qset$ for $i=1,\ldots,a$ and
$p_i<p_{i+1}$.  If $\theta\mid\floor{\theta/2^{p_i-1}}$ for each
$i=1,\ldots,a-1$ is all ones after discarding the leading zeros,
$a\ge2$, and $p_a=q-r$ (which implies that $r\geq1$), then if
(a)~$q=2r+1$ or (b)~$q>2r+1$ and $1\leq p_{a-1}\leq r$, then
$S^\Qset_n$ is the number of compositions of $n+q$ into parts 1,
$q+1$, $q+1+p_i$ for $i=1,\ldots,a-1$, and $q+1+j(q-r)$ for
$j\in\Zset^+$.
\end{theorem}

\begin{figure} 
\begin{center}
\includegraphics[width=11cm]{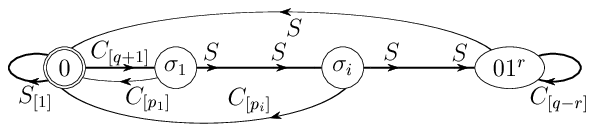}
\end{center}
\caption{Digraph for tiling a board with squares and combs
  corresponding to $\Qset$ specified in
  Theorem~\ref{T:compmin1arc} \cite{All-Comb}. In this figure and the next,
   node label $\sigma_i$ denotes the bit string corresponding to
filling the first $i-1$ empty cells of the $\Qset$-comb with squares and
discarding the leading 1s.}
\label{f:dgmin1arc}
\end{figure}

\begin{proof}
The digraph (Fig.~\ref{f:dgmin1arc}) generates the
metatiles $S$, $C_{[q+1]}S^{a}$,
$C_{[q+1]}S^{i-1}C_{[p_i]}$ for $i=1,\ldots,a-1$, and
$C_{[q+1]}S^{a-1}C_{[q-r]}^jS$ for $j\in\Zset^+$. Their respective
lengths are 1, $q+1$, $q+1+p_i$, and $q+1+j(q-r)$. As
$q-r=p_a>p_{a-1}$, the lengths are all distinct.
\end{proof}

The OEIS sequences to which Theorem~\ref{T:compmin1arc} applies along with the
corresponding $\Qset$ are as follows:
\seqnum{A224809} \{2,4\},
\seqnum{A375981} \{1,4,5\},
\seqnum{A375982} \{2,3,5\},
\seqnum{A375983} \{2,4,5\}.

\section{Connections with bit strings}\label{s:bits}
Others have noted that the number of length-$n$ bit strings (or
\textit{binary words}) having no two 1s that have positions in the
string that differ by $q$ (or, equivalently, are separated from one
another by $q-1$ digits) is $S^{\{q\}}_n$ \cite{KL91b}. This can be
generalized by considering the following bit string $s$ representing
subset $\Sset$: the $j$-th bit from the right in $s$ is 1 if and only
if $j\in\Sset$. (This representation is used in the program in the
Appendix.) Then it is clear that
$S^\Qset_{n,k}$ is the number of length-$n$
bit strings that contain $k$ 1s placed in such a way that the
difference in positions of any two 1s does not equal an element of
$\Qset$. We now mention an instance of this that corresponds to
\seqnum{A130137} in the OEIS \cite{Slo-OEIS}. If one instead places
the same restrictions on 0s rather than 1s, for the $\Qset=\{1\}$
case, the bit strings (i.e., those containing no 00) are called
\textit{Fibonacci binary words}, so named as the number of such words
of length $n$ is $F_{n+2}$. \seqnum{A130137} is described as the
number of length-$n$ Fibonacci binary words that do not contain the
substring 0110. This is $S^{\{1,3\}}_n$ since the three other possible
disallowed substrings that only prevent the occurrence of two 0s whose
positions differ by 3 (i.e., 0000, 0010, and 0100) do not need to be
mentioned; they each contain 00 and so cannot be present in Fibonacci
binary words.

\subsection{Binary word equivalence classes}

The final bijection we present concerns equivalence classes for the
occurrence of a subword within binary words. We say that two binary
words of the same length $n$ are \textit{equivalent with respect to a
  subword $\omega$} if $\omega$ occurs in the same position(s) in
those words. One can represent the equivalence class by the set of
subword positions, counting the leftmost position as 1. For example,
all binary words of the form $xx10010010xx$, where each $x$ can be 0
or 1, belong to the same equivalence class with respect to the subword
$10010$ and this class is represented by the set $\{3,6\}$.

\begin{theorem}\label{T:bs}
Let $\omega$ be a length-$l$ binary subword. We construct set $\Qset$
as follows: $j\in\Qset$ iff $\floor{\omega/2^j}\neq\omega\bmod2^{l-j}$
for $j=1,\ldots,l-1$.  If $q$ is the largest element of $\Qset$ (and
taking $q$ as zero when $\Qset=\{\}$) and $q=l-1$, then the number of
equivalence classes of binary words of length $n$ with respect to
subword $\omega$ is $S^\Qset_{n-q}$. Furthermore,
$S^\Qset_{n-q,k}$ is the number of such equivalence classes whose
representation as a set is of size $k$.
\end{theorem}
\begin{proof}
We show that the $k$-subsets of $\Nset_{n-q}$ satisfying the
conditions for disallowed differences specified by $\Qset$ are the
same as the sets representing the equivalence classes of binary words
of length $n$ in which $\omega$ appears $k$ times.  In general, the
possible positions of any subword of length $l$ in a word of length
$n$ (and therefore the possible elements of the set representing an
equivalence class) are $1,\ldots,n-l+1=n-q$.  From a restricted subset
$\Sset$ of $\Nset_{n-q}$ we construct an equivalence class of words
as follows. Starting with a length-$n$ word of $x$s (each $x$
representing 0 or 1), for each $s\in\Sset$ we place the start of
subword $\omega$ at position $s$ in the word.  Note that
$\floor{\omega/2^j}$ gives the $l-j$ leftmost bits of $\omega$ and
$\omega\bmod2^{l-j}$ gives the $l-j$ rightmost bits.  If the two are
equal (and so $j\notin\Qset$) then, on shifting the digits of subword
$\omega$ by $j$ places to the right and overlaying them on the
original subword, none of the digits of the original subword are
changed.  This means that we have constructed a valid equivalence
class once we have specified the restrictions on the remaining $x$s to
avoid any further appearances of $\omega$. Note that if there is an
$\omega$ at position $i$ and no other at $i+j$ for all
$j=1,\ldots,q$, we may place another $\omega$ at
position $i+l$ or after this. This coincides with the fact that
$l=q+1$ or any larger integers are not contained in the set $\Qset$ of
disallowed shifts.  The reversible nature of the construction
establishes the bijection.
\end{proof}

In the OEIS, the following sequences are connected with
Theorem~\ref{T:bs} (the corresponding $\Qset$ and the subword
starting with a 1 with the lowest numerical value it applies to
are also given):
\seqnum{A000079} ($\{\}$, 1),
\seqnum{A000045} ($\{1\}$, 10),
\seqnum{A000930} ($\{1,2\}$, 100),
\seqnum{A003269} ($\{1,2,3\}$, 1000),
\seqnum{A130137} ($\{1,3\}$, 1010),
\seqnum{A003520} ($\{1,2,3,4\}$, 10000),
\seqnum{A317669} ($\{1,2,4\}$, 10010; it was via this OEIS sequence that
the connection to restricted combinations was made),
\seqnum{A005708} ($\{1,2,3,4,5\}$, 100000),
\seqnum{A375185} ($\{1,2,3,5\}$, 100010),
\seqnum{A375186} ($\{1,2,4,5\}$, 100110),
\seqnum{A177485} ($\{1,3,5\}$, 101010).
Note that if the bits of subword $\omega$ are flipped (i.e., 0 changed to 1
and vice versa) and/or reversed, the corresponding $\Qset$ remains unchanged.

In the rest of this section we show that the elements of $\Qset$
concerned with the bijection given in Theorem~\ref{T:bs} are always a
well-based sequence and hence the generating function for the number
of equivalence classes and the recursion relation for the number of
equivalence classes whose set representation is of size $k$ are
straightforward to obtain.  The sequence $q_1,q_2,\ldots,q_m$, where
$q_i<q_{i+1}$ for all $i$, is said to be a \textit{well-based sequence}
if $q_1=1$ and for all 2-partitions of $q_j$, for $j=2,\ldots,q_m$,
at least one of the two parts is a $q_i$
\cite{Kit06,Val11}. Thus, for example, $1,3,4$ is not a well-based
sequence since $4=2+2$ and 2 is not in the sequence.
In the
present context, we can regard $q_i$ as being the elements of $\Qset$
with $q_m=q$, its largest element.  In the proof of
Theorem~\ref{T:wbbs}, we use the following definition, which is
equivalent aside from also classing the empty set as being well based:
the sequence is well based if $a=\abs{\Nset_q\setminus\Qset}$ is zero or
if, for all $i,j=1,\ldots,a$ (where $i$ and $j$ can be equal),
$p_i+p_j\notin\Qset$, where the $p_i$ are the elements of
$\Nset_q\setminus\Qset$.  E.g., the only well-based sequences of length
3 are the elements of the sets $\{1,2,3\}$, $\{1,2,4\}$, $\{1,2,5\}$,
and $\{1,3,5\}$.

\begin{theorem}\label{T:wbbs}
If $\Qset$ is constructed from subword $\omega$ as described in
Theorem~\ref{T:bs}, then the elements of $\Qset$ are a well-based
sequence with the additional condition that $\Lambda\notin\Qset$ where
$\Lambda$ is any linear combination (with integer coefficients) of the
$p_i$.
\end{theorem}
\begin{proof}
Suppose $\omega$ is the length-$l$ bit string $b_lb_{l-1}\cdots
b_2b_1$.  Then $p_i,p_j\notin\Qset$ iff $b_n=b_{n-p_i}$ and
$b_n=b_{n-p_j}$ for $n\leq l$ with $n\geq p_i+1$ and $n\geq p_j+1$,
respectively. Hence we also have $b_n=b_{n-p_i-p_j}$ for
$1+p_i+p_j\leq n\leq l$. This implies that $p_i+p_j\notin\Qset$ and
hence the elements of $\Qset$ form a well-based sequence.
In an analogous way, it is easily seen that
we have $b_n=b_{n-\Lambda}$ where $\Lambda$ is any integer-coefficient
linear combination of the $p_i$ such that $0<\Lambda<n$
and hence $\Lambda\notin\Qset$.
\end{proof}
The additional property in Theorem~\ref{T:wbbs} means, for example,
that the sets $\{1,2,5\}$ (for which $p_1=3$, $p_2=4$) and
$\{1,2,3,5,7\}$ (for which $p_1=4$, $p_2=6$) do not correspond to any
$\omega$ (since in both cases $p_2-p_1\in\Qset$) although their
elements form well-based sequences.

The following corollary gives a generating function for 
$E^{(\omega)}_{n,k}$, the number of equivalence
classes of length-$n$ binary words with respect to a length-$l$
subword $\omega$ that have a set representation of size $k$.

\begin{corollary}\label{C:wbbs}
The generating function $\tilde{g}(x,y)$, whose coefficient of
$x^ny^k$ equals $E^{(\omega)}_{n,k}$, is given by
\begin{equation}\label{e:wbbs} 
\tilde{g}(x,y)=\frac{1-\bar{c}}{(1-x)c-xy},
\end{equation}
where $\bar{c}=\sum_{i=1}^ax^{p_i}y$, $a=\abs{\Nset_q\setminus\Qset}$,
$c=1+\sum_{i=1}^{\abs{\Qset}}x^{q_i}y$, the set
$\Qset$ is obtained from $\omega$ using the procedure in
Theorem~\ref{T:bs}, $q_i$ are the elements of $\Qset$ of which the
largest is $q$ (with $q=0$ when $\Qset=\{\}$), and $p_i$ are the
elements of $\Nset_q\setminus\Qset$, provided that $q=l-1$.
\end{corollary}

\begin{figure}[b] 
\begin{center}
\includegraphics[width=10.5cm]{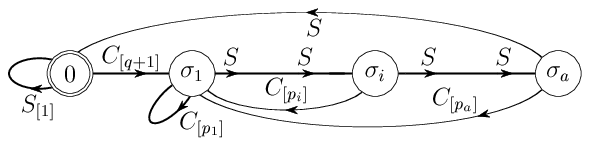}
\end{center}
\caption{Digraph for tiling a board with squares and $\Qset$-combs
  when the $q_i\in\Qset$ are a well-based sequence \cite{All-Comb}.}
\label{f:wb}
\end{figure}

\begin{proof}
We first deal with the $q=0$ case. Then $\Qset=\{\}$ and so
$\bar{c}=0$ and $c=1$.
This leaves $\tilde{g}(x,y)=1/(1-x-xy)$. In
this case, each binary word is in an equivalence class of its own and
so $E^{(\omega)}_{n,k}=\tbinom{n}{k}$ since this is the number of
length-$n$ binary words that contain $k$ 1s. The generating function
for $\tbinom{n}{k}$ is $1/(1-x-xy)$.  For $q>0$, since
$E^{(\omega)}_n=E^{(\omega)}_{n,0}=1$ for $0\leq n\leq q-1$, and, by
Theorem~\ref{T:bs}, $E^{(\omega)}_{n,k}=S^\Qset_{n-q,k}$ for $n\geq q$
and from Theorem~\ref{T:S=B} $S^\Qset_{n-q,k}=B^\Qset_{n,k}$ for
$n\geq q$, overall we have $E^{(\omega)}_{n,k}=B^\Qset_{n,k}$. Hence
$\tilde{g}(x,y)=G(x,y)$, the generating function such that the
coefficient of $x^ny^k$ is the number of tilings of an $n$ board that
contain $k$ combs.  The digraph for tiling a board with squares and
$\Qset$-combs when the elements of $\Qset$ are a well-based sequence
is shown in Fig.~\ref{f:wb} \cite{All-Comb}. The digraph has a common
node ($\sigma_1$) and the inner cycles ($S^{i-1}C_{[p_i]}$ for
$i=1,\ldots,a$) have lengths $p_i$. There is a single common circuit
($C_{[q+1]}S^a$), which is of length $q+1$, and one outer cycle
($S_{[1]}$), which is of length 1. The common circuit and all the
inner cycles contain a single comb. From Theorem~\ref{T:cn} we then
have
\begin{equation}\label{e:wbbgf}
G(x,y)=\frac{1-\bar{c}}{1-\bar{c}-x+x\bar{c}-x^{q+1}y}.
\end{equation}
Using the result that
$c+\bar{c}=1+\sum_{i=1}^qx^iy=1+(x-x^{q+1})y/(1-x)$ to simplify the
denominator gives \eqref{e:wbbs}.
\end{proof}

The following corollary, which is a generalization of a result
obtained previously \cite[Theorem~2]{Kit06} (and also Corollary~1 in
All24), is then easily obtained from \eqref{e:wbbgf} by applying
Corollary~\ref{C:snkgf}.

\begin{corollary}\label{C:wbgf}
If the elements $q_i$ of $\Qset$ are a well-based sequence then the
generating function such that the $x^ny^k$ coefficient thereof is
$S^\Qset_{n,k}$ is given by
\begin{equation}\label{e:wbgf} 
g(x,y)=\frac{c}{(1-x)c-xy},
\end{equation}
where $c=1+\sum_{i=1}^{\abs{\Qset}}x^{q_i}y$.
\end{corollary}

\section{Discussion}\label{s:dis}

Theorem~\ref{T:pcn}, which we derived here in order to obtain the
generating function for $S^{\{1,5\}}_n$, will be used elsewhere to
obtain expressions for generating functions for $S^\Qset_{n,k}$ in
terms of the elements of $\Qset$ for various other classes of
$\Qset$. As $\abs{\Nset_q\setminus\Qset}$ increases, the digraphs
become too complicated for the application of general results
analogous to Theorem~\ref{T:pcn}. In such instances, the repeated use
of Lemmas \ref{L:par}, \ref{L:ser}, and \ref{L:loop} yields the
generating function, as we demonstrate in \S\ref{s:srp1} for the
$\Qset=\{1,6\}$ case, although this becomes laborious for larger
digraphs. Then a more direct approach is to employ a transfer matrix
technique. For a digraph with $v$ nodes, the basic method gives the
following expression for the generating function \cite{Sta=11}:
\begin{equation}\label{e:tmgf}
  G(x,y)=\frac{\det (\mat{I}_{v-1}-\mat{\check{T}})}{\det (\mat{I}_v-\mat{T})},
\end{equation}
where $\mat{I}_m$ is the $m\times m$ identity matrix and the
$(i,j)$-th entry of transfer matrix $\mat{T}$ is the generating
function for the direct connection from the $j$-th node to the $i$-th
node. For convenience, one insists that the 0 node is the first
node. Then $\mat{\check{T}}$ is obtained from $\mat{T}$ by removing
the first column and first row. For example, choosing the $0^21$ and
$01$ nodes to be, respectively, the second and third nodes, the
transfer matrix for the digraph in Figure~\ref{f:dg14} is
\[
\mat{T}=\begin{pmatrix}x&0&1\\x^5y&0&x^3y\\0&1+x^2y&0\end{pmatrix},
\]
and \eqref{e:tmgf} gives
$G(x,y)=(1-x^3y-x^5y^2)/(1-x-x^3y+x^4y-x^5(y+y^2)+x^6y^2-x^7y^2)$. As
expected, this reduces to the first term in parentheses in
\eqref{e:g14w} on replacing $y$ by 1. More sophisticated techniques
involve first reducing the dimensions of the transfer matrices used in
an analogous formula to \eqref{e:tmgf} \cite{Klo09}.

Given the simple form of \eqref{e:tmgf}, one might ask why we have
gone to the trouble of introducing a lot of terminology concerning
digraphs and deriving Theorems~\ref{T:cn} and \ref{T:pcn}. There are
two reasons. The first is that the determinant of the transfer matrix
of a general digraph with an unspecified number of nodes (such as
those shown in Figs.~\ref{f:dgmin1arc} and \ref{f:wb} in this article
or Figs.~5 and 6 in All24)
is not trivial to simplify to give the
generating function; it is much easier to read off the lengths of the
cycles and circuits in the digraph and plug them into the formula for
$G(x,y)$. Once we have a general expression for $G(x,y)$ for a class
of $\Qset$, there is no longer any need to draw any digraphs (and deal
with transfer matrices) for an instance of that class.
Second, having concepts such as inner and outer cycles and common
nodes can lead to quick proofs, as is the case for Lemma~\ref{L:I1I2}
in the present work and the main theorem in a forthcoming article
\cite{AE25}.

\section{Acknowledgments}
The author thanks the two anonymous referees for their
time and the detailed and helpful comments. He also thanks Alex S\'aiz
for testing the program on his 64-bit machine.

\section*{Appendix: C program for finding $S^\Qset_n$ and $S^\Qset_{n,k}$}
There are two main features that contribute to the efficiency of the
algorithm used in the program listed below. Rather than directly
computing whether the difference of each pair $x,y$ of elements of
subset $\Sset$ is in $\Qset$, we perform a bitwise AND
operation on bit string representations of $\Sset$ shifted $x$
places to the right and $\Qset$. If the AND operation gives an answer of 0
then $y-x$ does not equal any element of $\Qset$ for all $y>x$.
We also exploit the result that if $\Sset$ is an allowed subset
of $\Nset_n$ then it is an allowed subset of $\Nset_m$ for
any $m>n$.

\lstset{language=C,
  frame=none,
  keepspaces=true,
  columns=fullflexible,
  basicstyle=\footnotesize\ttfamily,
  commentstyle=\color{webbrown},
  showstringspaces=false}
\begin{lstlisting}
// Anything after // on the same line is a comment.
// Program name: rcl.c  (restricted combinations on a line).
// Purpose: counts subsets of {1,2,...,n} such that 
// no two elements have a difference equal to an element of Q
// and finds S_n or S_{n,k} for n from 0 to MAXn, which is typically 32 or 64.
// To create an executable called rcl using the GNU C compiler enter:
// gcc rcl.c -O2 -o rcl  

// include libraries
#include <stdio.h>
#include <stdlib.h>
#include <string.h>
// sizeof( ) returns the number of bytes used by a variable type
// the variable type unsigned long int typically uses 4 bytes
// on 32-bit machines and 8 bytes on 64-bit machines
// hence MAXn is typically 32 or 64, respectively
#define MAXn 8*sizeof(unsigned long int)

// after including the libraries and defining MAXn, program starts here
// argc equals the number of command-line arguments plus 1
// argv is an array of character strings containing those arguments 
int main(int argc,char **argv) {
// start by defining the variables used by the program
// j-th bit from the right in the bit string is 1 iff j is in the set it represents
  unsigned long int Q=0,s,ss; // bit strings giving Q, subset, and shifted subset
  char *p; // used for reading in Q (given as a comma-separated list) from command line
  unsigned short int tri, // true if S_{n,k} triangle rather than S_n is wanted
    i,j,k=0,test,lastn=0;
  unsigned long int S[MAXn+1][MAXn+2]; // S[n][k]=S_{n,k} if tri true; if not, S[n][0]=S_n
// end of definitions of variables
// give examples of usage if program run with no arguments  
  if (argc==1) { // rcl called with no arguments
// puts( ) displays a string; \t is tab \n is new line 
    puts("example usage:\t rcl 1,2,4\n\t\t rcl 1,2,4 t");
    return 1; // end the program
  } // end if
// p starts by pointing to the character before the start of the 1st argument
  for (p=argv[1]-1;p;p=strchr(p,',')) { // loop to read in Q from command line argument
// line(s) within these braces only executed if p is not zero
// ++p means add 1 to p first
// atoi( ) converts a string into an integer, stopping as soon as it reaches a non-digit 
// x<<m means shift x in binary m places to the left
// | is the bitwise OR operation
    Q=Q|1<<(atoi(++p)-1); // add to existing Q
  } // end of loop: p is updated to point to the next comma in argument or zeroed if none
// && means a logical AND
  tri=(argc>2 && argv[2][0]=='t'); // true if argument 2 in command line is t
// initialize array that will contain S_{n,k}
// S_{n,0} is set to 1 in order to count the empty set
  for (i=0;i<=MAXn;i++) {
// cond?x:y returns x if cond is true and y otherwise
    for (j=0;j<=(tri?i+1:1);j++) {
      S[i][j]=j?0:1; // set k=0 totals to 1, zero k>0 totals
    }
  }
// ~0 means all ones in binary 
  for (s=1;s<~0;s++) { // loop to test all possible nonempty subsets s
    if (tri) {
// if S_{n,k} wanted, get k = number of set bits of s
      for (k=0,ss=s;ss;ss>>=1) {
        if (ss&1) k=k+1; // & is the bitwise AND operation 
      }          
    }
    for (ss=s,i=1;i<=MAXn;i++) { // test subset s
      test=ss&1; // will only test shifted subset against Q if rightmost bit of ss is 1
      ss>>=1; // shift shifted subset ss to the right by 1 
      if (ss==0) { // s is ok (no disallowed differences found)
        // i is now minimum n for which s is ok
        if (i>lastn) { // total(s) for previous n finished
          for (j=0;S[lastn][j];j++) {
            printf("%ld,",S[lastn][j]); // display S_n or S_{n,k} if nonzero
          }
          if (tri) puts(""); // new line 
          else fflush(stdout); // display immediately
          lastn=i; // update lastn for next time it is used
        }
        for (j=i;j<=MAXn;j++) {
          S[j][k]++; // if s is ok for a given n, it is ok for all higher n
        }
        break; // done with current s
      }
      if (test && (Q&ss)) break; // disallowed difference found - done with current s
    }
  }
  for (j=0;S[lastn][j];j++) {
    printf("%ld,",S[lastn][j]); // lastn now equals MAXn
  }
  puts(""); // print a new line at the end
  return 0;
}
\end{lstlisting}

\bigskip
\hrule
\bigskip

\noindent 2020 {\it Mathematics Subject Classification}:
Primary 05A15;
Secondary 05A19, 05B45, 05A05, 05C20, 11B39.

\noindent \emph{Keywords}:
combinatorial proof, $n$-tiling,
directed pseudograph,
restricted combination.

\bigskip
\hrule
\bigskip

\sloppy
\noindent (Concerned with sequences 
\seqnum{A000045},
\seqnum{A000079},
\seqnum{A000930},
\seqnum{A003269},
\seqnum{A003520},
\seqnum{A005708},
\seqnum{A005709},
\seqnum{A005710},
\seqnum{A006498},
\seqnum{A006500},
\seqnum{A011973},
\seqnum{A031923},
\seqnum{A079972},
\seqnum{A102547},
\seqnum{A121832},
\seqnum{A130137},
\seqnum{A177485},
\seqnum{A224809},
\seqnum{A224810},
\seqnum{A224811},
\seqnum{A224812},
\seqnum{A224813},
\seqnum{A224814},
\seqnum{A224815},
\seqnum{A259278},
\seqnum{A263710},
\seqnum{A276106}, 
\seqnum{A317669},
\seqnum{A322405},
\seqnum{A329146},
\seqnum{A351874},
\seqnum{A368244},
\seqnum{A374737},
\seqnum{A375185}, 
\seqnum{A375186},
\seqnum{A375981},
\seqnum{A375982},
\seqnum{A375983},
\seqnum{A375985}, 
\seqnum{A376033}, and
\seqnum{A385870}.)
\bigskip
\hrule
\bigskip

\vspace*{+.1in}
\noindent
Received August 31 2024; 
revised versions received September 2 2024;
February 14 2025; July 14 2025; July 17 2025.
Published in {\it Journal of Integer Sequences}, July 19 2025.

\bigskip
\hrule
\bigskip

\noindent
Return to \href{https://cs.uwaterloo.ca/journals/JIS/}{Journal of Integer Sequences home page}.
\vskip .1in

\end{document}